\documentclass[a4paper,11pt,leqno,twoside]{article}
\usepackage{amssymb,amsbsy,amsmath,amsfonts,amssymb,amscd,
graphics,color,footmisc,fancyhdr,multicol,fancybox,
graphicx,mathrsfs,rotating,ifthen,wasysym}

\usepackage{times}

\usepackage[dvipsnames,svgnames,x11names]{xcolor}
\usepackage[toc,page]{appendix}
\usepackage{pdfpages}
\usepackage{multirow}
\usepackage{nccrules}
\usepackage{textcomp}
\usepackage{comment}
\usepackage{framed}
\usepackage[all]{xy}
\usepackage{graphicx}
\usepackage{pgf,tikz}
\usepackage{mathrsfs}
\usetikzlibrary{arrows}
\usepackage{lipsum}
\usepackage{mathtools}

\definecolor{black}{cmyk}{1.,1.,1.,1.0}
\definecolor{blue}{cmyk}{1.,1.,0.,0.63}
\definecolor{green}{cmyk}{1.,0.,1.,0.63}

\newcommand{\medcap}{\mathbin{\scalebox{1.5}{\ensuremath{\cap}}}}
\newcommand{\medcup}{\mathbin{\scalebox{1.5}{\ensuremath{\cup}}}}
\usepackage[backend=bibtex,url=false,doi=false,eprint=false]{biblatex}
\usepackage{hyperref}
\hypersetup{
       unicode=true,          
     pdftoolbar=true,        
    pdfmenubar=true,        
    pdffitwindow=false,     
    pdfstartview={FitH},    
       pdfproducer={}, 
    pdfcreator={pdfLaTeX},   
    pdfproducer={}, 
    pdfnewwindow=true,      
    colorlinks=false,       
    linkcolor=blue,          
    citecolor=green,        
    filecolor=magenta,      
    urlcolor=cyan           
}
\usepackage[left=2cm,right=2cm,top=2cm,bottom=1.8cm]{geometry}
\usepackage{amssymb,amsmath,amsthm}

\DeclareMathOperator{\reg}{reg}
\DeclareMathOperator{\pr}{pr}
\DeclareMathOperator{\dif}{d}

\DeclareMathOperator{\supp}{supp}
\DeclareMathOperator{\ord}{ord}

\newtheorem{thm}{Theorem}[section]
\newtheorem{lem}{Lemma}[section]
\newtheorem{pro}{Proposition}[section]
\newtheorem{cor}{Corollary}[section]

\theoremstyle{plain}
\numberwithin{equation}{section}
\newcommand{\thistheoremname}{}
\newtheorem*{genericthm*}{\thistheoremname}
\newenvironment{namedthm*}[1]{\renewcommand{\thistheoremname}{#1}%
\begin{genericthm*}}
{\end{genericthm*}}

\newtheoremstyle{named}{}{}{\itshape}{}{\bfseries}{.}{.5em}{\thmnote{#3's }#1}
\theoremstyle{named}

\title{\bf Entire holomorphic curves into projective spaces
\\
intersecting a generic hypersurface of high degree}
\providecommand{\keywords}[1]{\textbf{\textit{Keywords:}} #1}
\providecommand{\subject}[1]{\textbf{\textit{Mathematics Subject Classification 2010:}} #1}

\author{Dinh Tuan HUYNH, Duc-Viet VU and Song-Yan XIE}

\newcommand{\Addresses}{{
\bigskip
\footnotesize
\textsc{Dinh Tuan Huynh, Department of Mathematics, Graduate School of Science, Osaka University, Toyonaka, Osaka 560-0043, Japan \& Department of Mathematics, College of Education, Hue University, 34 Le Loi St., Hue City, Vietnam}

\par\nopagebreak
\noindent
\textit{E-mail address}: \texttt{dinh-tuan.huynh@math.sci.osaka-u.ac.jp}
\newline

\noindent
\textsc{Duc-Viet Vu, Korea institute for advanced study,
 85 Hoegiro, Dongdaemun-gu, Seoul 02455, Republic of Korea}

\noindent
\par\nopagebreak
\noindent
\textit{E-mail address}: \texttt{vuviet@kias.re.kr}
\newline

\noindent
\textsc{Song-yan Xie, Max-Planck-Institut f\"{u}r Mathematik, Vivatsgasse 7, 53111 Bonn, Germany}

\par\nopagebreak
\noindent
\textit{E-mail address}: \texttt{songyan.xie@mpim-bonn.mpg.de}

}}
\date{\vspace{-5ex}}

\begin{document}

\maketitle
\begin{abstract}
In this note, we establish the following Second Main Theorem type estimate for every algebraically nondegenerate entire curve $f\colon\mathbb{C}\rightarrow\mathbb{P}^n(\mathbb{C})$, in presence of a {\sl generic} divisor $D\subset\mathbb{P}^n(\mathbb{C})$ of sufficiently high degree $d\geq 15(5n+1)n^n$: for every $r$ outside a subset of $\mathbb{R}$ of finite Lebesgue measure and every real positive constant $\delta$, we have
\[
T_f(r)
\leq
\,N_f^{[1]}(r,D)
+
O\big(\log T_f(r)\big)
+
\delta \log r,
\]
  where $T_f(r)$ and $N_f^{[1]}(r,D)$ stand for the order function and the $1$-truncated counting function in Nevanlinna theory. This inequality quantifies recent results on the logarithmic Green--Griffiths conjecture.
\end{abstract}
\keywords{Nevanlinna theory}, {Second Main Theorem}, {holomorphic curve}, {Green--Griffiths' conjecture}, {algebraic degeneracy}

\noindent
\subject{32H30}, {32A22}, {30D35}, {32Q45}, 
\section{Introduction and the main result}

We first recall the standard notation in Nevanlinna theory. Let $E=\sum_i\alpha_i\,a_i$ be a divisor on $\mathbb{C}$ where $\alpha_i\geq 0$, $a_i\in\mathbb{C}$ and let
$k\in \mathbb{N}\cup\{\infty\}$. Denote by $\Delta_t$ the disk $\{z\in\mathbb{C},|z|<t\}$. Summing the $k$-truncated degrees of the divisor on disks by
\[
n^{[k]}(t,E)
:=
\sum_{a_i\in\Delta_t}
\min
\,
\{k,\alpha_i\}
\eqno
{{\scriptstyle (t\,>\,0)},}
\]
the \textsl{truncated counting function at level} $k$ of $E$ is then defined by taking the logarithmic average
\[
N^{[k]}(r,E)
\,
:=
\,
\int_1^r \frac{n^{[k]}(t, E)}{t}\,\dif\! t
\eqno
{{\scriptstyle (r\,>\,1)}.}
\]
When $k=\infty$, we write $n(t,E)$, $N(r,E)$ instead of $n^{[\infty]}(t,E)$, $N^{[\infty]}(r,E)$. Let $f\colon\mathbb{C}\rightarrow \mathbb{P}^n(\mathbb{C})$ be an entire curve having a reduced representation $f=[f_0:\cdots:f_n]$ in the homogeneous coordinates $[z_0:\cdots:z_n]$ of $\mathbb{P}^n(\mathbb{C})$. Let $D=\{Q=0\}$ be a divisor in $\mathbb{P}^n(\mathbb{C})$ defined by a homogeneous polynomial $Q\in\mathbb{C}[z_0,\dots,z_n]$ of degree $d\geq 1$. If $f(\mathbb{C})\not\subset D$, we define the \textsl{truncated counting function} of $f$ with respect to $D$ as
\[
N_f^{[k]}(r,D)
\,
:=
\,
N^{[k]}\big(r,(Q\circ f)_0\big),
\]
where $(Q\circ f)_0$ denotes the zero divisor of $Q\circ f$.

The \textsl{proximity function} of $f$ for the divisor $D$ is defined as
\[
m_f(r,D)
\,
:=
\,
\int_0^{2\pi}
\log
\frac{\big\Vert f(re^{i\theta})\big\Vert^d\,
\Vert Q\Vert}{\big|Q(f)(re^{i\theta})\big|}
\,
\frac{\dif\!\theta}{2\pi},
\]
where $\Vert Q\Vert$ is the maximum  absolute value of the coefficients of $Q$ and
\[
\big\Vert f(z)\big\Vert
\,
=
\,
\max
\,
\{|f_0(z)|,\ldots,|f_n(z)|\}.
\]
Since $\big|Q(f)\big|\leq \Vert Q\Vert\cdot\Vert f\Vert^d$, one has $m_f(r,D)\geq 0$.

Lastly, the \textsl{Cartan order function} of $f$ is defined by
\begin{align*}
T_f(r)
\,
:&=
\,
\frac{1}{2\pi}\int_0^{2\pi}
\log
\big\Vert f(re^{i\theta})\big\Vert \dif\!\theta\\
\,
&=
\,
\int_1^r\dfrac{\dif\! t}{t} \int_{\Delta_t}f^*\omega_n
+
O(1),
\end{align*}
where $\omega_n$ is the Fubini--Study form on $\mathbb{P}^n(\mathbb{C})$.

With the above notations, the Nevanlinna theory consists of two fundamental theorems (for a comprehensive presentation, see Noguchi-Winkelmann \cite{Noguchi-Winkelmann2014}).

\begin{namedthm*}{First Main Theorem}\label{fmt} Let $f\colon\mathbb{C}\rightarrow \mathbb{P}^n(\mathbb{C})$ be a holomorphic curve and let $D$ be a hypersurface of degree $d$ in $\mathbb{P}^n(\mathbb{C})$ such that $f(\mathbb{C})\not\subset D$. Then for every $r>1$, the following holds
\[
m_f(r,D)
+
N_f(r,D)
\,
=
\,
d\,T_f(r)
+
O(1),
\]
whence
\begin{equation}
\label{-fmt-inequality}
N_f(r,D)
\,
\leq
\,
d\,T_f(r)+O(1).
\end{equation}
\end{namedthm*}

Hence the First Main Theorem gives an upper bound on the counting function in terms of the order function. On the other side, in the harder part, so-called Second Main Theorem, one tries to establish a lower bound for the sum of certain counting functions. Such types of estimates were given in several situations.

Throughout this note, for an entire curve $f,$ the notation $S_f(r)$ means a real function of $r \in \mathbb{R}^+$ such that there is a constant $C$ for which 
$$S_f(r) \le C \, T_f(r)+ \delta \log r$$
for every positive constant $\delta$ and every $r$ outside of a subset (depending on $\delta$) of finite Lebesgue measure of $\mathbb{R}^+$. 

A holomorphic curve $f\colon\mathbb{C}\rightarrow \mathbb{P}^n(\mathbb{C})$ is said to be \textsl{algebraically (linearly) nondegenerate} if its image is not contained in any hypersurface (hyperplane). A family of $q\geq n+1$ hypersurfaces $\{D_i\}_{1\leq i\leq q}$ in $\mathbb{P}^n(\mathbb{C})$ is in {\sl general position} if any $n+1$ hypersurfaces in this family have empty intersection:
\[
\underset{i\in I}{\medcap}\,\supp(D_i)=\emptyset \eqno
{\scriptstyle (\forall\,I\,\subset\,\{1,\dots,q\},\, |I|=n+1).}
\]

We recall here the following classical result \cite{cartan1933}, with truncation level $n$.
 
\begin{namedthm*}{Cartan's Second Main Theorem}\label{smt_cartan}
Let $f\colon\mathbb{C}\rightarrow \mathbb{P}^n(\mathbb{C})$ be a linearly nondegenerate holomorphic curve and let $\{H_i\}_{1\leq i\leq q}$ be a family of $q>n+1$ hyperplanes in general position in $\mathbb{P}^n(\mathbb{C})$. Then the following estimate holds
\[
(q-n-1)
\,
T_f(r)
\,
\leq
\,
\sum_{i=1}^q N_f^{[n]}(r,H_i)+S_f(r).
\]
\end{namedthm*}

In the one-dimensional case, Cartan recovered the classical Nevanlinna theory for nonconstant meromorphic functions and families of $q>2$ distinct points. Since then, many author tried to extend the result of Cartan to the case of (possible) nonlinear hypersurface. Eremenko-Sodin \cite{Eremenko-Sodin1992} established a Second Main Theorem for $q>2n$ hypersurfaces $D_i$ ($1\leq i\leq q$) in general position in $\mathbb{P}^n(\mathbb{C})$ and for any nonconstant holomorphic curve $f\colon\mathbb{C}\rightarrow \mathbb{P}^n(\mathbb{C})$ whose image is not contained in $\cup_{1\leq i\leq q}\supp(D_i)$. Keeping the same assumption on $q>n+1$ hypersurfaces, Ru \cite{Minru2004} proved a stronger estimate for algebraically nondegenerate holomorphic curves $f\colon\mathbb{C}\rightarrow \mathbb{P}^n(\mathbb{C})$. He then extended this result to the case of algebraically nondegenerate holomorphic mappings into an arbitrary nonsingular complex projective variety \cite{Minru2009}. Note that it remains open the question of truncating the counting functions in the above generalizations of Cartan's Second Main Theorem. Some results in this direction are obtained recently but one requires the presence of more targets, see for instance \cite{An-Quang-Thai2013}, \cite{Thai-Viet2013}.

In the other context, Noguchi-Winkelmann-Yamanoi \cite{Noguchi-Winkelmann-Yamanoi2002} established a Second Main Theorem for algebraically nondegenerate holomorphic curves into semiabelian varieties intersecting an effective divisor. Yamanoi \cite{Yamanoi2004} obtained a similar result in the case of abelian varieties with the best truncation level $1$, which is extended to the case of semiabelian varieties by Noguchi-Winkelmann-Yamanoi \cite{Noguchi-Yamanoi-Winkelmann2008}.

In the qualitative aspect, the (strong) Green-Griffiths conjecture stipulates that if $X$ is a complex projective space of general type, then there exists a proper subvariety $Y\subsetneq X$ containing the image of every nonconstant entire holomorphic curve $f\colon\mathbb{C}\rightarrow X$.

Following a beautiful strategy of Siu \cite{Siu2004}, Diverio, Merker and Rousseau \cite{DMR2010} confirmed this conjecture for generic hypersurface $D \subset\mathbb{P}^{n+1}$ of degree $d\geq 2^{n^5}$. Berczi \cite{Berczi2015} improved the degree bound to $d\geq n^{9n}$. Demailly \cite{Demailly2012} gave a new degree bound
\[
d
\geq
\frac{n^4}{3}\bigg(
n\log(n\log(24n))
\bigg)^n.
\]

In the logarithmic case, namely for the complement of a hypersurface $D\subset\mathbb{P}^n(\mathbb{C})$, there is another variant of this conjecture, so-called the logarithmic Green-Griffiths conjecture, which expects that for a generic hypersurface $D \subset\mathbb{P}^n(\mathbb{C})$ having degree $d\geq n+2$, there should exist a proper subvariety $Y\subset\mathbb{P}^n(\mathbb{C})$ containing the image of every nonconstant entire holomorphic curve $f\colon\mathbb{C}\rightarrow\mathbb{P}^n
(\mathbb{C})\setminus D$. Darondeau \cite{Darondeau2016} gave a positive answer for this case with effective degree bound
\[
d
\geq
(5n)^2n^n.
\]

In this note, we show that the current method towards the Green-Griffiths conjecture can yield {\em not only qualitative but also quantitative result}, namely a Second Main Theorem type estimate in presence of {\em only one generic} hypersurface $D$ of sufficiently high degree with the truncation level $1$.

\begin{namedthm*}{Main Theorem}
Let $D \subset \mathbb{P}^{n}(\mathbb{C})$ be a generic divisor having degree 
\[
d
\geq
15(5n+1)n^n
.
\]
Let $f\colon\mathbb{C}\rightarrow\mathbb{P}^{n}(\mathbb{C})$ be an entire holomorphic curve. If $f$ is algebraically nondegenerate, then the following estimate holds
\[
T_f(r)
\leq
\,
N_f^{[1]}(r,D)
+
S_f(r).
\]
\end{namedthm*}

For background and standard techniques in Nevanlinna theory, we use the book of Noguchi-Winkelmann \cite{Noguchi-Winkelmann2014} as our main reference. The proof of the existence of logarithmic jet differentials in the last part of this note is based on the work of Darondeau \cite{Darondeau2016}.

\section*{Acknowledgments}
The authors would like to thank Jo\"{e}l Merker for his encouragements and his comments that greatly improved the manuscript. We would like to thank Nessim Sibony for very fruitful discussions on the paper \cite{Siboni-Paun2014}. We want to thank Junjiro Noguchi, Katsutoshi Yamanoi, Yusaku Tiba and Yuta Kusakabe for their interests in our work and for listening through many technical details. We would like to thank the referee for his/her careful reading of the manuscript and helpful suggestions. The first author is supported by the fellowship of the Japan Society for the Promotion of Science and the Grant-in-Aid for JSPS fellows Number 16F16317.

\section{Logarithmic jet differentials}
\subsection{Logarithmic Green-Griffiths $k$-jet bundle}

The general strategy to prove the logarithmic Green-Griffiths conjecture consists of two steps. The first one is to produce many algebraically independent differential equations that all holomorphic curve $f\colon\mathbb{C}\rightarrow \mathbb{P}^n(\mathbb{C})\setminus D$ must satisfy. The second step consists in producing enough jet differentials from an initial one such that from the corresponding algebraic differential equations, one can eliminate all derivative in order to get purely algebraic equations.

The central geometric object corresponding to the algebraic differential equations is the logarithmic Green-Griffiths $k$-jet bundle constructed as follows. Let $X$ be a complex manifold of dimension $n$. For a point $x\in X$, consider the holomorphic germs $(\mathbb{C},0)\rightarrow (X,x)$. Two such germs are said to be equivalent if they have the same Taylor expansion up to order $k$ in some local coordinates around $x$. The equivalence class of a holomorphic germ $f\colon (\mathbb{C},0)\rightarrow (X,x)$ is called the {\sl $k$-jet of $f$}, denote $j_k(f)$. A $k$-jet $j_k(f)$ is said to be {\sl regular} if $f'(0)\not=0$. For a point $x\in X$, denote by $j_k(X)_x$ the vector space of all $k$-jets of holomorphic germs $(\mathbb{C},0)\rightarrow (X,x)$. Set
\[
J_k(X)
:=
\underset{x\in X}{\medcup}\,J_k(X)_x
\]
and consider the natural projection
\[
\pi_k\colon J_k(X)\rightarrow X.
\]
Then $J_k(X)$ is a complex manifold which carries the structure of a holomorphic fiber bundle over $X$, which is called the {\sl $k$-jet bundle over $X$}. When $k=1$, $J_1(X)$ is canonically isomorphic to the holomorphic tangent bundle $T_X$ of $X$.

For an open subset $U\subset X$, for a section $\omega\in H^0(U,T_X^*)$, for a $k$-jet $j_k(f)\in J_k(X)|_U$, the pullback $f^*\omega$ is of the form $A(z)dz$ for some holomorphic function $A$. Since each derivative $A^{(j)}$ ($0\leq j\leq k-1$) is well defined, independent of the representation of $f$ in the class $j_k(f)$, the holomorphic $1$-form $\omega$ induces the holomorphic map
\begin{equation}
\label{trivialization-jet}
\tilde{\omega}
\colon
J_k(X)|_U
\rightarrow
\mathbb{C}^k;\,\,j_k(f)\rightarrow
\big(A(z),A(z)^{(1)},\dots,A(z)^{(k-1)}
\big).
\end{equation}
Hence on an open subset $U$, a local holomorphic coframe $\omega_1\wedge\dots\wedge\omega_n\not=0$ yields a trivialization $H^0(U, J_k(X))\rightarrow U\times(\mathbb{C}^k)^n$ by giving new $n k$ independent coordinates:
\[
\sigma\rightarrow(\pi_k\circ\sigma;\tilde{\omega}_1\circ\sigma,\dots,\tilde{\omega}_n\circ\sigma),
\]
where $\tilde{\omega}_i$ are defined as in \eqref{trivialization-jet}. The components $x_i^{(j)}$ ($1\leq i\leq n$, $1\leq j\leq k$) of $\tilde{\omega}_i\circ\sigma$ are called {\sl jet-coordinates}.
In a more general setting where $\omega$ is a section over $U$ of the sheaf of meromorphic $1$-forms, the induced map $\tilde{\omega}$ is meromorphic.

Now, in the logarithmic setting, let $D\subset X$ be a normal crossing divisor on $X$. This means that at each point $x\in X$, there exist some local coordinates $z_1,\dots,z_{\ell},z_{\ell+1},\dots,z_n$ ($\ell=\ell(x)$) centered at $x$ in which $D$ is defined by
\[
D=
\{z_1\dots z_{\ell}=0
\}.
\]
Following Iitaka \cite{Iitaka1982}, the {\sl logarithmic cotangent bundle of $X$ along $D$}, denoted by $T_X^*(\log D)$, corresponds to the locally-free sheaf generated by
\[
\dfrac{\dif\!z_1}{z_1},\dots,\dfrac{\dif\! z_{\ell}}{z_{\ell}},z_{\ell +1},\dots,z_n
\]
in the above local coordinates around $x$.

A holomorphic section $s\in H^0(U,J_k(X))$ over an open subset $U\subset X$ is said to be a {\sl logarithmic $k$-jet field} if $\tilde{\omega}\circ s$ are holomorphic for all sections $\omega\in H^0(U',T_X^*(\log D))$, for all open subsets $U'\subset U$, where $\tilde{\omega}$ are induced maps defined as in \eqref{trivialization-jet}. Such logarithmic $k$-jet fields define a
subsheaf of $J_k(X)$, and this subsheaf is itself a sheaf of sections of a holomorphic fiber bundle over $X$, called the {\sl logarithmic $k$-jet bundle over $X$ along $D$}, denoted by $J_k(X,-\log D)$ \cite{Noguchi1986}.

The group $\mathbb{C}^*$ acts fiberwise  on the jet bundle  as follows. For local coordinates 
\[
z_1,\dots,z_{\ell},z_{\ell+1},\dots,z_n\eqno\scriptstyle{(\ell=\ell(x))}
\]
centered at $x$ in which $D=\{z_1\dots z_{\ell}=0\}$, for any logarithmic $k$-jet field along $D$ represented by some germ $f=(f_1,\dots,f_n)$, if $\varphi_{\lambda}(z)=\lambda z$ is the homothety with ratio $\lambda\in \mathbb{C}^*$, the action is given by
\[
\begin{cases}
\big(\log(f_i\circ \varphi_{\lambda})\big)^{(j)}
=
\lambda^j
\big(\log f_i\big)^{(j)}\circ\varphi_{\lambda} &
\quad
\scriptstyle{(1\,\leq\,i\,\leq\,\ell),}
\\
\big(f_i\circ \varphi_{\lambda}\big)^{(j)}
\quad\quad\,=
\lambda^j f_i^{(j)}\circ\varphi_{\lambda}
&
\quad
\scriptstyle{(\ell+1\,\leq\,i\,\leq\,n).}
\end{cases}
\]

Now we are ready to introduce the {\sl Green-Griffiths $k$-jet bundle} \cite{Green-Griffiths1980} in the logarithmic setting. A logarithmic jet differential of {\sl order} $k$ and {\sl degree} $m$ at a point $x\in X$ is a polynomial $Q(f^{(1)},\dots,f^{(k)})$ on the fiber over $x$ of $J_k(X,-\log D)$ enjoying weighted homogeneity:
\[
Q(j_k(f\circ\varphi_{\lambda}))
=
\lambda^m
Q(j_k(f)).
\]
Denote by $E_{k,m}^{GG}T_X^*(\log D)_x$ the vector space of such polynomials and set
\[
E_{k,m}^{GG}T_X^*(\log D)
:=
\underset{x\in X}{\medcup}\,
E_{k,m}^{GG}T_X^*(\log D)_x.
\]
By Fa\`{a} di bruno's formula \cite{Constantine1996}, \cite{Merker2015}, $E_{k,m}^{GG}T_X^*(\log D)$ carries the structure of a vector bundle over $X$, called {\sl logarithmic Green-Griffiths vector bundle}. A global section $\mathscr{P}$ of $E_{k,m}^{GG}T_X^*(\log D)$ locally is of the following type in jet-coordinates $x_i^{(j)}$:
{\footnotesize
\[
\underset{|\alpha_1|+2|\alpha_2|+\dots+k|\alpha_k|=m}
{\sum_{\alpha_1,\dots,\alpha_k\in\mathbb{N}^n}}
A_{\alpha_1,\dots,\alpha_k}
\bigg(
\prod_{i=1}^{\ell}
\big(
(\log x_i)^{(1)}
\big)^{\alpha_{1,i}}
\prod_{i=\ell+1}^{n}
\big(
(x_i)^{(1)}
\big)
^{\alpha_{1,i}}
\bigg)
\dots
\bigg(
\prod_{i=1}^{\ell}
\big(
(\log x_i)^{(k)}
\big)^{\alpha_{k,i}}
\prod_{i=\ell+1}^{n}
\big(
(x_i)^{(k)}
\big)
^{\alpha_{k,i}}
\bigg),
\]
}
where
\[
\alpha_{\lambda}
=
(\alpha_{\lambda,1},\dots,\alpha_{\lambda,n})
\in\mathbb{N}^n
\eqno
\scriptstyle{(1\,\leq\,\lambda\,\leq\, k)}
\]
are multi-indices of length 
\[
|\alpha_{\lambda}|
=
\sum_{1\leq i\leq n}
\alpha_{\lambda,i},
\]
and where $A_{\alpha_1,\dots,\alpha_k}$ are locally defined holomorphic functions.

By the following classical result, the first step to prove the Green-Griffiths conjecture reduces to finding logarithmic jet differentials valued in the dual of some ample line bundle. 

\begin{namedthm*}{Fundamental vanishing theorem} (\cite{Demailly1997}, \cite{Dethloff-Lu2001})
Let $X$ be a smooth complex projective variety and let $D\subset X$ be a normal crossing divisor on $X$. If $\mathscr{P}$ is a nonzero global holomorphic logarithmic jet differential along $D$ vanishing on some ample line bundle $\mathscr{A}$ on $X$, namely if
\[
0\not=\mathscr{P}
\in
H^0\big(
X,
E_{k,m}^{GG}T_X^*(\log D)
\otimes
\mathscr{A}^{-1}
\big),
\]
then all nonconstant holomorphic curves $f:\mathbb{C}\rightarrow X\setminus D$ must satisfy the associated differential equation
\begin{equation}
\label{alg.diff.equation}
\mathscr{P}
\big(j_k(f)
\big)
\equiv 0.
\end{equation}
\end{namedthm*}

In the compact case, the existence of such global sections has been proved recently, first by Merker \cite{Merker2015} for the case of smooth hypersurfaces of general type in $\mathbb{P}^n(\mathbb{C})$, and later for arbitrary general projective variety by Demailly \cite{Demailly2011}. Adapting this technique in the logarithmic setting, Darondeau \cite{Darondeau2016} obtained a similar result for smooth hypersurface in projective space, provided that the degree is high enough compared with the dimension.

\begin{pro} (\cite[Th. 1.2]{Darondeau2016})
\label{existence-logarimith-jet-differential-lionel}
Let $c\in\mathbb{N}$ be a positive integer and let $D\subset\mathbb{P}^n(\mathbb{C})$ be a smooth hypersurface having degree
\[
d
\geq
15(c+2)
\,
n^n.
\]
For jet order $k=n$, for weighted degrees $m\gg d$
big enough, the vector space of global logarithmic jet differentials along $D$ of order $k$ and weighted degree $m$ vanishing on the $m$-th tensor power of the ample line bundle $\mathcal{O}_{\mathbb{P}^n(\mathbb{C})}(c)$ has positive dimension:
\[
\dim
\,
H^0
\big(
\mathbb{P}^n(\mathbb{C}),
E_{n,m}^{GG}T_{\mathbb{P}^n(\mathbb{C})}^*(\log D)
\otimes
\mathcal{O}_{\mathbb{P}^n(\mathbb{C})}(c)^{-m}
\big)
>
0.
\]
\end{pro}

\section{Second Main Theorem for logarithmic jet differential}

Let $D\subset\mathbb{P}^n(\mathbb{C})$ be a smooth hypersurface in $\mathbb{P}^n(\mathbb{C})$. Let $f:\mathbb{C}\rightarrow\mathbb{P}^n(\mathbb{C})$ be an entire holomorphic curve, not necessary in the complement of $D$. If there exists a global logarithmic jet differential $\mathscr{P}$ which does not satisfy \eqref{alg.diff.equation}, then the fundamental vanishing theorem guarantees that the curve $f$ must intersect the hypersurface $D$. Furthermore, in the quantitative aspect, based on the proof of the fundamental vanishing theorem, it is known that a Second Main Theorem type estimate
\[
T_f(r)
\leq
C\,N_f(r,D)
+
S_f(r)
\]
can be deduced from the existence of such global section $\mathscr{P}$. There are several variants of the above estimate, see for instance in \cite{Siboni-Paun2014}, \cite{Siu2015}. Here we provide more information about the constant $C$ and truncation of the counting function.

Before going to introduce the main result of this section, we need to recall the following lemma on logarithmic derivative which is a crucial tool  in Nevanlina theory. 

\begin{namedthm*}{Logarithmic derivative Lemma}
Let $g\not\equiv 0$ be a nonzero meromorphic function on $\mathbb{C}$. For any integer $k\geq 1$, we have
\[
m_{\frac{g^{(k)}}{g}}
(
r
)
:=
m_{\frac{g^{(k)}}{g}}
(
r,\infty
)
=
S_g(r).
\]
\end{namedthm*}

We refer to \cite[Lem. 4.7.1]{Noguchi-Winkelmann2014} for a more general version of the above Lemma. Here is our main result in this section.

\begin{thm}
\label{smt-form-logarithmic-diff-jet}
Let $f\colon\mathbb{C}\rightarrow\mathbb{P}^n
(\mathbb{C})$ be an entire curve and let $D\subset\mathbb{P}^n(\mathbb{C})$ be a smooth hypersurface. Let $\tilde{m}$ be a positive integer. If there exists a global logarithmic jet differential
\[
\mathscr{P}
\in
H^0
\big(
\mathbb{P}^n(\mathbb{C}),
E_{k,m}^{GG}T_{\mathbb{P}^n(\mathbb{C})}^*(\log D)\otimes\mathcal{O}_{\mathbb{P}^n(\mathbb{C})}(1)^{- \widetilde{m}}
\big)
\]
 such that
\[
\mathscr{P}
\big( j_k(f)
\big)
\not\equiv 0,
\]
then the following Second Main Theorem type estimate holds:
\[
T_f(r)
\leq
\frac{m}{\widetilde{m}}\,N^{[1]}_f(r,D)
+
S_f(r).
\]
\end{thm}

\begin{proof}
Our  proof is partly  based on \cite[Lem. 4.7.1]{Noguchi-Winkelmann2014} and \cite{Demailly1997}. Let $Q$ be the irreducible homogeneous polynomial defining $D$. By assumption, $\mathscr{P}
\big(
j_k(f)
\big)$ is a nonzero meromorphic section of $f^* \mathcal{O}_{\mathbb{P}^n(\mathbb{C})}(1)^{-\widetilde{m}}$. Let $D_{\mathscr{P},f}$ be the  pole divisor of $\mathscr{P}\big(j_k(f)\big)$.

Let $\big(V,\phi)$ be a small enough local chart  of $\mathbb{P}^n(\mathbb{C})$ such that $\phi: \mathbb{P}^n(\mathbb{C}) \rightarrow \mathbb{C}^n$ is a rational map and $D$ is given by $D=\{z_1=0\},$ where $z=(z_1, \cdots,z_n)$ are the natural coordinates on $\mathbb{C}^n$. Put 
\begin{align} \label{def_fVj}
f_{j}:= \phi(f),
\end{align}
which is a \emph{meromorphic function} on $\mathbb{C}$ for $1 \le j \le n$. Then $f$ is written in the local chart $V$ as $(f_{1}, \cdots, f_{n})$ on $f^{-1}(V)$.  Observe that  $f_{1}/ Q(f)$ is a nowhere vanishing holomorphic function on $f^{-1}(V)$. Recall that on $V$, the section $\mathscr{P}\big(j_k(f)\big)$ can be written as
\begin{align} \label{def_P}
\mathscr{P}\big(j_k(f)\big)
=
\underset{|\alpha_1|+2|\alpha_2|+\dots+k|\alpha_k|=m}
{\sum_{\alpha_1,\dots,\alpha_k\in\mathbb{N}^n}}
A_{\alpha_1,\dots,\alpha_k}
\prod_{\ell=1}^k
\bigg(
\big(
(\log f_{1})^{(\ell)}\big)^{\alpha_{\ell,1}}
\prod_{j=2}^{n}
\big(
f_{j}^{(\ell)} 
\big)
^{\alpha_{\ell,j}}
\bigg),
\end{align}
where $A_{\alpha_1,\dots,\alpha_k}$ are holomorphic functions on $f^{-1}(V)$ and $\alpha_j=(\alpha_{j,1}, \cdots, \alpha_{j,n})$ for $1 \le j \le n$. Hence, the support of $D_{\mathscr{P},f}$ is a subset of the zero set of $Q\circ f$ on $\mathbb{C}$. Furthermore, since for each $1 \leq \ell \leq k$, the pole order of $(\log f_{1})^{(\ell)}$ at any point $z\in\mathbb{C}$ is at most $\ell \min\{\ord_z f_{1}, 1\}$ (hence at most $\ell \min\{\ord_z Q(f), 1\}$) and since the degree of $\mathscr{P}$ is $m$, we get
\[
D_{\mathscr{P},f}
\leq 
m \sum_{z\in\mathbb{C}}\min\{\ord_z(Q\circ f), 1\}
z.
\]
Let $h$ be the pullback by $f$ of the Fubini-Study form $\omega_n$ on $\mathbb{P}^n(\mathbb{C})$. Using the Poincar\'e-Lelong formula, we have
\[
\dif\!\dif\!^c \log \| \mathscr{P}\big(j_k(f)\big)\|_h \ge \widetilde{m} f^* \omega - [D_{\mathscr{P},f}],
\]
where $[D_{\mathscr{P},f}]$ is the integration current of $D_{\mathscr{P},f}$. Combining this fact with the above inequality, we obtain
\[
\widetilde{m}\, T_{f}(r)
+
O(1)
\leq
\int_{\partial\Delta_r}\log
\|\mathscr{P}\big(j_k(f)\big)\|_h^2
\,
\dfrac{\dif\!\theta}{2\pi}
+
m\, N^{[1]}_f(r,D).
\]
Thus, it remains to verify 
\begin{equation}
\label{first-reduction}
\int_{\partial \Delta_r}\log
\| \mathscr{P}\big(j_k(f)\big) \|_h^2
\,
\dfrac{\dif\!\theta}{2\pi}
=
S_f(r).
\end{equation}
Using a partition of unity on $\mathbb{P}^n(\mathbb{C})$, the problem reduces to proving that
\begin{equation}
\label{second-reduction}
\int_{ \partial \Delta_r}\log |\chi(f)\mathscr{P}\big(j_k(f)\big)|^2 \,
\dfrac{\dif\!\theta}{2\pi}
=
S_f(r),
\end{equation}
where $\chi$ is a smooth positive function  compactly supported on a local chart $V$ as above. Using the following elementary observations with $s,s_1, \cdots, s_N \in \mathbb{R}^*_+$:
\begin{align*}
\log s
&=
\log^+ s-\log^+\frac{1}{s} \le \log^+ s \\
\log^+\sum_{i=1}^N s_i
&\leq
\sum_{i=1}^N\log^+ s_i
+
\log N\\
\log^+\prod_{i=1}^Ns_i
&\leq
\sum_{i=1}^N\log^+ s_i,
\end{align*}
where $\log^+$ denotes $\max\{\log, 0\}$, we get
\begin{align}
\label{third-reduction}
\int_{ \partial \Delta_r}
\log |\chi(f)\mathscr{P}
\big(j_k(f)\big)|^2
\,\dfrac{\dif\!\theta}{2\pi}
&\leq 
\underset{|\alpha_1|+2|\alpha_2|+\dots+k|\alpha_k|=m}
{\sum_{\alpha_1,\dots,\alpha_k\in\mathbb{N}^n}}
\sum_{\ell=1}^k
\bigg(
\int_{\partial \Delta_r}
\log^+
\big(
\chi(f)
 |(\log f_1)^{(\ell)}|^{\alpha_{\ell,1}}
\big)
\,\dfrac{\dif\!\theta}{2\pi}\notag\\
&+
\sum_{j=2}^n
\int_{\partial \Delta_r}
\log^+
\big(
\chi(f) |f_j^{(\ell)}|^{\alpha_{\ell,j}}
\big)
\,
\dfrac{\dif\!\theta}{2\pi}
\bigg)
+
O(1),
\end{align}
Recall from (\ref{def_fVj}) that $f_j$ are meromorphic functions on $\mathbb{C}$ for $1 \le j \le n$. Hence applying the logarithmic derivative Lemma to $f_1$, we infer that
\begin{align} \label{ine_loga1nhe}
\int_{\partial \Delta_r}
\log^+
\big(
\chi(f)
|(\log f_1)^{(\ell)}|^{\alpha_{\ell,1}}
\big)
\,\dfrac{\dif\!\theta}{2\pi}
=
S_f(r).
\end{align}
Therefore, it suffices to show that this property still holds for the remaining terms in the right-hand side of \eqref{third-reduction}. Continuing to apply the logarithmic derivative Lemma, we obtain
\[
\int_{\partial \Delta_r}
\log^+
\big(
\chi(f) |f^{(\ell)}_j|^{\alpha_{\ell,j}}
\big)
\,
\dfrac{\dif\!\theta}{2\pi}
\leq
 c \int_{ \partial\Delta_r}
\log^+
\big(
\chi(f) |f^{(1)}_j|^2
\big)
\,
\dfrac{\dif\!\theta}{2\pi}
+
S_f(r),
\]
for some constant $c$ which is independent of $r,f$. Hence it remains to check
\[
\int_{\partial \Delta_r}
\log^+
\big(
\chi(f) |f^{(1)}_j|^{2}
\big)
\,
\dfrac{\dif\!\theta}{2\pi}
=
S_f(r).
\]
This can be done by using the similar arguments as in \cite[p. 149]{Noguchi-Winkelmann2014}. For the reader's convenience, we present the idea here. Since $\chi$ is compactly supported on $V$, there exists a bounded positive function $B$ for which $\chi \dif\! z_j \wedge \dif\!\bar{z}_j \le B(z)\, \omega_n$ on $V$ for $2 \le j \le n$. This yields
\[
\chi(f)| f^{(1)}_j|^2 \dif\! z \wedge \dif\!\bar{z}
=
f^*(\chi\, \dif\! z_j \wedge \dif\!\bar{z}_j)
\leq
B(f) f^*\omega_n.
\]
The pullback $f^* \omega_n$ is of the form $B_1 \dif\! z \wedge \dif\! \bar{z}$. Hence we deduce from the above inequality that 
\begin{align*}
\int_{\partial \Delta_r} \log^+
\big(
\chi(f) |f^{(1)}_j|^{2}
\big)
\,
\dfrac{\dif\!\theta}{2\pi} &\leq \frac{1}{2\pi}\int_{\partial\Delta_r} \log^+ B(f) \,
\dif\!\theta
+
 \frac{1}{2\pi} \int_{\partial \Delta_r} \log^+ B_1 \,
\dif\!\theta  \\
&\le  \frac{1}{2\pi}\int_{ \partial \Delta_r} \log^+ B_1 \,
\dif\!\theta
+
\supp_{z \in V} |B(z)|.
\end{align*}
Estimating $\int_{\partial \Delta_r} \log^+ B_1$ is done by following the same arguments as in the proof of the logarithmic derivative Lemma, see \cite[(3.2.8)]{Noguchi-Winkelmann2014}. The proof is finished.
\end{proof}

\section{Existence of logarithmic jet differentials}

Let $f: \mathbb{C}\rightarrow\mathbb{P}^n(\mathbb{C})$ be an algebraically nondegenerate holomorphic curve. Following the second step in Siu's strategy to prove the Green-Griffith conjecture, morally, if we can produce enough logarithmic jet differentials valued in the dual of some ample line bundle on $\mathbb{P}^n(\mathbb{C})$, then among them, we can choose {\it at least one} such that $f$ {\it does not satisfy} the algebraic differential equation \eqref{alg.diff.equation}.

\begin{thm} \label{the_existence_jet}
\label{construction-jet-differential} 
Let $c$ be a positive integer with $c\geq 5n-1$. Let $D\subset\mathbb{P}^n(\mathbb{C})$ be a generic smooth hypersurface in $\mathbb{P}^n(\mathbb{C})$ having degree
\[
d
\geq
15(c+2)n^n.
\]
Let $f\colon\mathbb{C}\rightarrow\mathbb{P}^n
(\mathbb{C})$ be an entire holomorphic curve. If $f$ is algebraically nondegenerate, then for jet order $k=n$ and for weighted degrees $m>d$ big enough, there exists an integer $0 \le \ell \le m$ and a global logarithmic jet differential
\[
\mathscr{P}
\in
H^0
\big(
\mathbb{P}^n(\mathbb{C}),
E_{n,m}^{GG}T_{\mathbb{P}^n(\mathbb{C})}^*(\log D)
\otimes
\mathcal{O}_{\mathbb{P}^n(\mathbb{C})}(1)^{-cm+ \ell(5n-2)}
\big)
\]
such that
\begin{align}
\label{eq_canPfjets}
\mathscr{P}\big(j_n(f)
\big)
\not\equiv
0.
\end{align}
\end{thm}

The rest of this section is devoted to proving Theorem \ref{the_existence_jet} whose proof is based on \cite{DMR2010,Darondeau2016}. Let $\mathbf{S}:= \mathbb{P} H^0\big(\mathbb{P}^n(\mathbb{C}), \mathcal{O}(d)\big)$ be the projective parameter space of homogeneous polynomials of degree $d$ in $\mathbb{P}^n(\mathbb{C})$ which identifies with the projective space $\mathbb{P}^{\mathbf{N}_d}(\mathbb{C})$ of dimension
\[
\mathbf{N}_d
=
\dim\,
\mathbb{P}H^0
\big(
\mathbb{P}^n(\mathbb{C}),\mathcal{O}_{\mathbb{P}^n(\mathbb{C})}(d)
\big)
=
\left(\begin{matrix}
n+d\\
d
\end{matrix}\right)-1.
\]
We then introduce the {\sl universal hypersurface}
\[
\mathcal{H}
\subset
\mathbb{P}^n(\mathbb{C})
\times \mathbf{S}
\]
parametrizing all hypersurfaces of fixed degree $d$ in $\mathbb{P}^n(\mathbb{C})$, defined by the equation
\[
0
=
\sum_{\alpha\in\mathbb{N}^{n+1}}A_{\alpha}\, Z^{\alpha}
\]
in the following two collections of homogeneous coordinates
\begin{align*}
Z&=[Z_0:\dots:Z_n]\in\mathbb{P}^n(\mathbb{C}),\\ A&=[(A_{\alpha})_{\alpha\in\mathbb{N}^{n+1},|\alpha|=d}]\in\mathbb{P}^{\mathbf{N}_d}(\mathbb{C}),
\end{align*}
where $\alpha=(\alpha_0,\dots,\alpha_n)\in\mathbb{N}^{n+1}$ are multiindices. Since $\mathcal{H}$ is a smooth hypersurface on $\mathbb{P}^n(\mathbb{C}) \times \mathbf{S}$, we can construct the logarithmic $k$-jet bundle
\[
J_k(\mathbb{P}^n(\mathbb{C})
\times
\mathbf{S},-\log \mathcal{H})
\]
over $\mathbb{P}^n(\mathbb{C})\times \mathbf{S}$ along $\mathcal{H}$.

Now, let $\eta$ be  be the natural projection from $J_k(\mathbb{P}^n(\mathbb{C})\times
\mathbf{S},-\log \mathcal{H})$ to $\mathbb{P}^n(\mathbb{C})\times \mathbf{S}$. Let $\pr_1$ and $\pr_2$ be the natural projections from $\mathbb{P}^n(\mathbb{C})\times \mathbf{S}$ to the first and second part, respectively. Let $\mathcal{V}_{\mathcal{H},k}$ be the analytic subset of $J_k(\mathbb{P}^n(\mathbb{C})\times \mathbf{S},-\log \mathcal{H})$ consisting of all {\sl vertical logarithmic jet fields of order $k$} which, by definition, are jets $j_k(f)$ such that $f$ lies entirely in some fiber of the second projection $\pr_2$.

Denote by $\mathcal{V}^{\reg}_{\mathcal{H},k}$ the open subset consisting of all regular jets. By \cite[p. 571-572]{Darondeau_vectorfield} (see also \cite[p. 1088]{Merker2009}), $\mathcal{V}_{\mathcal{H},k}^{\reg}$ is smooth manifold. Following the method of producing new jet differentials developed by Siu \cite{Siu2004}, in the logarithmic setting, one needs to construct low pole order frames on $\mathcal{V}^{\reg}_{\mathcal{H},k}$.

\begin{pro}(\cite[Main Theorem]{Darondeau_vectorfield})
\label{slanted-vector-field}
For jet order $k\geq 1$, for degree $d\geq k$, the twisted tangent bundle
\[
T_{\mathcal{V}_{\mathcal{H},k}}
\otimes
\eta^*
\big(
\mathcal{O}_{\mathbb{P}^n(\mathbb{C})}(5\,k-2)\otimes\mathcal{O}_{\mathbf{S}}(1)
\big)
\]
is generated over $\mathcal{V}_{\mathcal{H},k}^{\reg}\setminus\eta^{-1}\mathcal{H}$ by its global holomorphic sections.
\end{pro}

In fact, those global sections mentioned in the above Proposition are global vector fields on the whole logarithmic $k$-jet bundle and satisfy the canonical tangential conditions described as in \cite{Darondeau_vectorfield, Merker2009}. Hence they are true vector fields on the smooth part of $\mathcal{V}_{\mathcal{H},k}$. Moreover by the constructions in \cite{Darondeau_vectorfield, Merker2009,Siu2015}, the coefficients of those vector fields are polynomials in local logarithmic jet coordinates.   

Let $\mathbf{Z}_0$ be the subset of $\mathbf{S}$ consisting of all $s$ whose corresponding hypersurface $D_s$ is not smooth. Observe that $\mathbf{Z}_0$ is a proper analytic subset of $\mathbf{S}$.

From now on we work with the fixed jet order $k=n$. Since $\pr_2^{-1}s= \mathbb{P}^n(\mathbb{C})\times \{s\}$ and since $D_s$ is smooth for every $s$ outside $\mathbf{Z}_0$, one can define $J_n(\pr_2^{-1}s, -\log D_s)$ for any $s \in \mathbf{S} \backslash \mathbf{Z}_0$. Let us set
\begin{align*}
\mathcal{L}&:= \bigcup_{s\in \mathbf{S} \backslash \mathbf{Z}_0 } J_n( \pr_2^{-1}s, -\log D_s),\\
\mathcal{L}^{\reg}&:= \bigcup_{s\in \mathbf{S} \backslash \mathbf{Z}_0 }  J^{\reg}_n(\pr_2^{-1}s, -\log D_s).\\
\end{align*}
Observe that $\mathcal{L}$ has a natural structure of holomorphic fiber bundle over $\mathbb{P}^n(\mathbb{C})\times (\mathbf{S} \backslash \mathbf{Z}_0)$. Note also that $\mathcal{L},\mathcal{L}^{\reg}$ are open subsets of $\mathcal{V}_{\mathcal{H},n}$ and $\mathcal{V}_{\mathcal{H},n}^{\reg}$, respectively. Set
\[\mathcal{E}:= \bigcup_{s\in \mathbf{S} \backslash \mathbf{Z}_0 } E^{GG}_{n,m}T^*\pr_2^{-1}s(\log D_s),
\]
then $\mathcal{E}$ carries the structure of holomorphic vector bundle over $\mathbb{P}^n(\mathbb{C})\times (\mathbf{S} \backslash \mathbf{Z}_0)$. This fact allows us to
extend holomorphically a nonzero jet differential provided by Proposition \ref{existence-logarimith-jet-differential-lionel}. Let us enter the details.

\begin{lem}
\label{extend-holomorphically-section}
Let $c\geq 5n-1$ be a positive integer. For degree $d\geq 15(c+2) n^n$, for weighted degree $m \gg d$, there exists a proper analytic subset $\mathbf{Z}$ of $\mathbf{S}$ containing $\mathbf{Z}_0$ such that for every $s \in \mathbf{S} \setminus \mathbf{Z}$, we can find a Zariski open neighborhood $U_s$ of $s$ in $\mathbf{S} \backslash \mathbf{Z}_0$ and a nonzero holomorphic section $\mathscr{P} \not \equiv 0$ of the twisted vector bundle $\mathcal{E} \otimes \pr_1^* \mathcal{O}_{\mathbb{P}^n(\mathbb{C})}(1)^{-c m}$ over $\pr_2^{-1}U_s$. 
\end{lem}

\begin{proof}
By construction, for any $s\in \mathbf{S} \backslash \mathbf{Z}_0,$ the restriction of $\mathcal{E}$ to $\pr_2^{-1}s=\mathbb{P}^n(\mathbb{C}) \times \{s\}$ coincides with 
\[
E^{GG}_{n,m}T^*\pr_2^{-1}s(\log D_s).
\]
Hence Proposition \ref{existence-logarimith-jet-differential-lionel} guarantees the existence of a nonzero global section 
\[
0\not \equiv\mathscr{P}_s
\in
H^0
\big(
\mathbb{P}^n(\mathbb{C}) \times \{s\},\mathcal{E} \otimes \pr_1^* \mathcal{O}_{\mathbb{P}^n(\mathbb{C})}(1)^{-c m}|_{\pr^{-1}_2 s}
\big)
\]
of the restriction of the twisted vector bundle $\mathcal{E} \otimes \pr_1^* \mathcal{O}_{\mathbb{P}^n(\mathbb{C})}(1)^{-c m}$ to $\pr_2^{-1}s$. By the semi-continuity theorem (c.f. \cite[p. 50]{Mumford_alelian_varieties}), there exists a proper Zariski closed subset $\mathbf{Z}$ of $\mathbf{S}$  containing $\mathbf{Z}_0$ such that for any $s \in \mathbf{S} \setminus \mathbf{Z}$, the natural restriction map 
\[
H^0
\big(\pr_2^{-1} U_{s}, \mathcal{E} \otimes \pr_1^* \mathcal{O}_{\mathbb{P}^n(\mathbb{C})}(1)^{-c m}
\big)
\longrightarrow
H^0
\big(
\pr_2^{-1} s,\mathcal{E}  \otimes \pr_1^* \mathcal{O}_{\mathbb{P}^n(\mathbb{C})}(1)^{-c m}|_{\pr^{-1}_2 s}
\big)
\]
is onto for some Zariski open subset $U_s\subset\mathbf{S} \backslash \mathbf{Z}_0$ containing $s.$ As a consequence, the above section $\mathscr{P}_s$ can be extended holomorphically to a section $\mathscr{P}$ of $\mathcal{E} \otimes \pr_1^* \mathcal{O}_{\mathbb{P}^n(\mathbb{C})}(1)^{-c m}$ over $\pr_2^{-1} U_s$.
\end{proof}

\begin{proof}[Proof of Theorem~\ref{construction-jet-differential}.] 

Let $\mathbf{Z}, \mathscr{P}$ be as in Lemma \ref{extend-holomorphically-section}. Let us first describe precisely the generic assumption of $D$ in the statement. By this, we mean that if $D$ corresponds to the element $s\in\mathbf{S}$ (i.e. $D=D_s$), then $s$ lies outside $\mathbf{Z}\cup H_{\mathbf{S}}$, where $H_{\mathbf{S}}$ is a fixed arbitrary hyperplane of $\mathbf{S}$. Here the condition $s\not\in H_{\mathbf{S}}$ is given in order to get rid of $\eta^*\mathcal{O}_{\mathbf{S}}(1)$ in Proposition \ref{slanted-vector-field} because the line bundle $\mathcal{O}_{\mathbf{S}}(1)$ is trivial on $\mathbf{S} \backslash H_{\mathbf{S}}$. From now on, we fix $s\in\mathbf{S}\setminus \mathbf{Z}$ and $D=D_s$.

Applying Proposition \ref{slanted-vector-field} for jet order $k=n$, the twisted tangent bundle 
\[
T_{\mathcal{V}_{\mathcal{H},n}} \otimes \eta^*\big(\mathcal{O}_{\mathbb{P}^n(\mathbb{C})}(5n-2) \otimes \mathcal{O}_\mathbf{S}(1)\big)
\]
is generated by its global holomorphic sections over $\mathcal{V}_{\mathcal{H},n}^{\reg} \setminus \eta^{-1}\mathcal{H}$. Moreover, the coefficients of those sections are polynomials in the logarithmic $n$-jet coordinates associated with the canonical coordinates of $\mathbb{P}^n(\mathbb{C}) \times \mathbf{S}$. In particular, the restriction of the bundle 
\[
T_{\mathcal{V}_{\mathcal{H},n}} \otimes \eta^*\mathcal{O}_{\mathbb{P}^n(\mathbb{C})}(5n-2)
\]
to $\eta^{-1} Y$, where
\[
Y:=
\pr_2^{-1} (U_{s} \backslash H_{\mathbf{S}}) \setminus \mathcal{H}
\]
is generated on $(\mathcal{V}_{\mathcal{H},n}^{\reg} \cap \eta^{-1} Y)$  by its global sections whose coefficients are polynomials in the logarithmic jet coordinates as above.

For $0\leq \ell\leq m$, let $v_1, \dots, v_{\ell}$ be sections of $T_{\mathcal{V}_{\mathcal{H},n}} \otimes \eta^*\mathcal{O}_{\mathbb{P}^n(\mathbb{C})}(5n-2)$ over the open subset $\mathcal{L} \subset \mathcal{V}_{\mathcal{H},n}$. As explained below, the significance of those sections is that they allow to construct new global logarithmic jet differentials. Indeed,  we can view $\mathscr{P}$ as a holomorphic mapping
\[
\mathscr{P}
\colon
\mathcal{L}
|_{\pr_2^{-1} U_s}
\rightarrow
\pr_1^*
\mathcal{O}_{\mathbb{P}^n(\mathbb{C})}(1)^{-cm}
|_{\pr_2^{-1} U_s},
\]
which is locally a homogeneous polynomial of degree $m$. It follows that the Lie derivative $(v_1 \cdots v_{\ell})\cdot\mathscr{P}$ is also a holomorphic map from $\mathcal{L}
|_{\pr_2^{-1} U_s}$ to 
\[
\pr_1^*\mathcal{O}_{\mathbb{P}^n(\mathbb{C})}(1)^{-cm+\ell(5n -2)}
\]
and is locally a homogeneous polynomial of  the same degree $m$. The fact that the derivative of $\mathscr{P}$ along $v_j$ preserves the degree $m$ can be deduced from the fact that the coefficients of $v_j$ are polynomials in the logarithmic jet coordinates and by the chain rule of derivatives, the degree of those polynomials should compensate the losses of degree due to the differentiation with respect to $v_j,$ see \cite[Sec. 3.7]{Siu2015}.  In summary, we obtain a holomorphic map
\begin{align}
\label{state_daohamLie}
(v_1 \cdots v_{\ell})\cdot\mathscr{P}: \mathcal{L}
|_{\pr_2^{-1} U_s} \rightarrow \pr_1^*\mathcal{O}_{\mathbb{P}^n(\mathbb{C})}(1)|_{\pr_2^{-1} U_s}^{-cm+\ell(5n -2)}.
\end{align}
By composing $f$ with the inclusion 
\[
\mathbb{P}^n(\mathbb{C})
\hookrightarrow
\mathbb{P}^n(\mathbb{C})
\times
\{s\}
\subset
\mathbb{P}^n(\mathbb{C})
\times \mathbf{S},\]
we can consider $f$ as a holomorphic curve into $Y \subset \mathbb{P}^n(\mathbb{C}) \times \mathbf{S}$ because $s \in U_{s} \setminus H_{\mathbf{S}}$ and $f$ is not included in $D.$ Let $\{\mathscr{P}=0\} \subset \pr_2^{-1} U_s$ be the zero divisor of $\mathscr{P},$ where we view $\mathscr{P}$ as a holomorphic section of $\mathcal{E} \otimes \pr_1^* \mathcal{O}_{\mathbb{P}^n(\mathbb{C})}(1)^{-c m}$ over $\pr_2^{-1} U_s.$ Since $f$ is algebraically nondegenerate, there exists $z_0 \in \mathbb{C}$ such that $f'(z_0) \not =0$ and $f(z_0) \not \in D \cap \{\mathscr{P}=0\}$. Consequently, we get
\begin{align} \label{inclusioj_nf}
j_n(f)(z_0) \in (\mathcal{L}^{\reg} \cap \eta^{-1}Y).
\end{align}
Now proceeding as in \cite{DMR2010}, we can show that there exist global slanted vector fields $v_1,\dots, v_{\ell}$ for some $0 \le \ell \le m$ such that 
\[
(v_1 \cdots v_{\ell})\cdot\mathscr{P}\big(j_n(f)\big)
\not=
0.
\]
For reader's convenience, we briefly recall the idea. Denoted by $\mathscr{P}_s$ the restriction of $\mathscr{P}$ to $\mathbb{P}^n(\mathbb{C})\times \{s\}$.  Consider the logarithmic jet coordinates $(z, z^{(1)},\dots, z^{(n)}) \in \mathbb{C}^{n(n+1)}$ around $j_n(f)(z_0)$ of $\mathcal{L}|_{\mathbb{P}^n(\mathbb{C}) \times\{s\}}$. Using a linear change of coordinates, we obtain  modified logarithmic jet coordinates $(z', z'^{(1)},\dots, z'^{(n)})$ in which $j_n(f)(z_0)$ is the origin. Since $\mathscr{P}_s$ is locally a homogeneous polynomial in logarithmic jet coordinates whose coefficients are holomorphic functions on local charts of $\mathbb{P}^n(\mathbb{C})$, so it is in the new logarithmic jet coordinates.

By the choice of $z_0,$ there exists a coefficient $A_{\alpha_1, \dots,\alpha_n}(z)$ of $\mathscr{P}_s$ (see (\ref{def_P})) for which $A_{\alpha_1, \dots, \alpha_n}(f(z_0)) \not= 0$. Let
\[
A_{\alpha_1, \dots, \alpha_n}(z) \big(z'^{(1)}\big)^{\alpha_1} \cdots \big(z'^{(n)}\big)^{\alpha_n}
\]
 be the monomial of $\mathscr{P}_s$ associated with $A_{\alpha_1, \dots,\alpha_n}$, where $\alpha_j \in \mathbb{N}^n$ for $1 \le j \le n$ and $|\alpha_1| + 2 |\alpha_2| +\cdots+n |\alpha_n|\le m$. We then choose local vector fields $v'_1, \dots v'_{\ell}$ around the origin $j_n(f)(z_0)$ for which 
\[
(v'_1 \cdots v'_{\ell})\cdot \mathscr{P}\big(j_n(f)(z_0)\big)
=
A_{\alpha_1, \dots, \alpha_n}\big(f(z_0)\big) \not =0.
\] 
As we mentioned above, these vector fields $v'_1, \dots, v'_{\ell}$ can be generated by global vector fields $v_1, \dots, v_{\ell}$ on $(\mathcal{L}^{\reg} \cap \eta^{-1}Y)$. Combining this with (\ref{inclusioj_nf}), we get $(v_1 \cdots v_{\ell})\cdot \mathscr{P}\big(j_n(f)(z_0) \big)\not = 0$. This together with (\ref{state_daohamLie}) implies (\ref{eq_canPfjets}). The proof of Theorem \ref{the_existence_jet} is completed.
\end{proof}

\begin{cor}
Let $c$ be a positive integer with $c\geq 5n-1$. Let $D\subset\mathbb{P}^n(\mathbb{C})$ be a generic hypersurface in $\mathbb{P}^n(\mathbb{C})$ having degree
\[
d
\geq
15(c+2)n^n.
\]
Let $f\colon\mathbb{C}\rightarrow\mathbb{P}^n
(\mathbb{C})$ be an entire holomorphic curve. If $f$ is algebraically nondegenerate, then the following Second Main Theorem type estimate holds:
\[
T_f(r)
\leq
\dfrac{1}{c- 5n+2} \, N^{[1]}_f(r,D)+ S_f(r).
\]
In particular, choosing $c=5n-1$, one obtains the Main Theorem.
\end{cor}

\begin{proof} This is a direct application of  Theorem~\ref{smt-form-logarithmic-diff-jet} to global logarithmic jet differential $\mathscr{P}$ supplied by Theorem \ref{the_existence_jet}, where
\[
\tilde{m}=mc-\ell(5n-2)\geq m(c-5n+2)\geq m.
\]
\end{proof}

\bigskip
\newpage
\begin{center}
\printbibliography
\end{center}
\Addresses
\end{document}